\documentclass[12pt]{amsart}
\title{Noetherianity and Rooted Trees}
\author{Daniel Barter}
\date{\today}

\makeatletter
\def\@endtheorem{\endtrivlist}
\makeatother 

\usepackage{easyfig} 
\usepackage{float}

\usepackage{booktabs} 

\usepackage{amsmath,amssymb}
\usepackage{tikz-cd}
\usepackage{mathrsfs} 
\usepackage{tikz}
\usetikzlibrary{graphs} 

\usepackage[boxsize = 0.1cm]{ytableau} 

\usepackage{enumerate} 

\usepackage[margin=2.5cm]{geometry} 

\usepackage{color}

\theoremstyle{plain}
\newtheorem{Theorem}{Theorem}
\newtheorem{Corollary}[Theorem]{Corollary}

\theoremstyle{definition}
\newtheorem{Definition}[Theorem]{Definition}
\newtheorem{Lemma}[Theorem]{Lemma}
\newtheorem{Example}[Theorem]{Example}

\DeclareMathOperator{\Mor}{Mor}

\DeclareMathOperator{\In}{in}
\DeclareMathOperator{\Rep}{Rep}

\newcommand{\TT}{{\bf T}}
\newcommand{\FTT}{{\bf FT}}
\newcommand{\FPT}{{\bf FPT}}
\newcommand{\PT}{{\bf PT}}
\newcommand{\Vect}{{\bf Vec}}
\newcommand{\CC}{{\bf C}}
\newcommand{\FI}{{\bf FI}}
\newcommand{\FS}{{\bf FS}}

\newcommand{\up}[1]{\underset{\mbox{\tiny #1}}{(}}
\newcommand{\down}[1]{\underset{\mbox{\tiny #1}}{)}}

\begin{document}

\begin{abstract}
Let \( \TT \) be the category whose 
objects are rooted trees and morphisms 
are order embeddings preserving the 
root. We prove that finitely generated 
representations of \( \TT \) are
Noetherian using techniques developed 
by Sam and Snowden which generalize 
classical Gr\"{o}bner theory. The proof 
uses a relative version of Kruskal's 
tree Theorem.
\end{abstract}

\maketitle

\section{Introduction} \label{sec:introduction}

Let \( \CC \) be a category and \( 
\Vect \) the category of vector spaces 
over a field \( k \) with arbitrary 
characteristic. Write \( 
\Rep(\CC) \) for the category of 
functors \( \CC \to \Vect \). Such 
functors are called 
{\bf representations} of \( \CC \). Let 
\( \TT \) be the 
category whose objects are 
rooted trees 
and morphisms are order embeddings 
preserving the root. In this paper we 
shall prove
\begin{Theorem} 
\label{T_is_noetherian}
Finitely generated \( \TT 
\)-representations are Noetherian.
\end{Theorem}
Theorem \ref{T_is_noetherian} is proved 
using Gr\"{o}bner categories, first 
defined by Richter in \cite{gR86}, 
and further developed by Sam and 
Snowden in 
\cite{SS2014}. Gr\"{o}bner categories 
reduce Noetherianity questions to 
combinatorial questions. In all 
examples so far, the combinatorial 
questions reduce to Higman's lemma, or 
some variant. For the category \( \TT 
\), the combinatorial question reduces 
to Kruskal's tree Theorem.

\subsection{Motivation and previous 
work}

Theorem \ref{T_is_noetherian} is a 
generalization of theorem 
\ref{FI_is_noetherian}, 
which was proved independently by 
Church, Ellenberg and Farb in 
\cite{TCJEBF} and by Snowden in
\cite{AS2013}:

\begin{Theorem} 
\label{FI_is_noetherian}
Let \( \FI \) be the category of finite 
sets with injections. Then finitely 
generated \( \FI \)-representations 
over a field of characteristic \( 0 \) are 
Noetherian.
\end{Theorem}
Theorem \ref{FI_is_noetherian} has the 
following Corollary, due to 
Church, Ellenberg and Farb in \cite{TCJEBF}:
\begin{Corollary} 
\label{cor:cohomology_of_configuration_spaces}
Let \( M \) be a manifold and \( S \) a 
finite set. Then \( S \mapsto C_S(M)= 
\{ \text{injections \( S \to M \)} \} \) 
is a functor from \( \FI^{\rm op} \) 
into the category of manifolds, and \( 
S \mapsto H^d(C_S(M),\mathbb{Q}) \) is 
a finitely generated \( \FI 
\)-representation.
\end{Corollary}
 We hope that Corollary 
\ref{cor:cohomology_of_configuration_spaces} 
convinces the reader that Theorems 
\ref{T_is_noetherian} and 
\ref{FI_is_noetherian} are interesting.  
Motivated by Theorem 
\ref{FI_is_noetherian}, Sam and Snowden 
developed the theory of Gr\"{o}bner 
categories in \cite{SS2014}. They 
proved
\begin{Theorem} 
\label{quasi_grobner_implies_noeth}
Let \( \CC \) be quasi-Gr\"{o}bner 
category. Then every finitely generated 
\( \CC \)-representation is Noetherian.
\end{Theorem} 
Sam and Snowden also proved that the 
categories \( \FI_d, \FS^{\rm op},{\bf 
VA}, \FI_G, \FS^{\rm op}_G \) are 
quasi-Gr\"{o}bner. In all of these 
examples, the objects are 
parameterized by the natural numbers. 
The category \( \TT \) is the first 
known example of a quasi-Gr\"{o}bner 
category whose objects do not have a 
natural bijection with \( \mathbb{N}^p  
\).

\subsection{Open problems} 
This paper raises several questions:
\begin{enumerate}
\item Are there any interesting spaces 
which are acted upon by tree 
automorphism groups? If we could find 
non trivial functors from \( \TT^{\rm 
op} \) into the category of spaces, 
then Theorem \ref{T_is_noetherian} 
might imply results like Corollary 
\ref{cor:cohomology_of_configuration_spaces}.
\item If \( V \) is a finitely 
generated \( \TT \)-representation, 
what can one say about the function \( 
T \mapsto \dim V_T \)? If \( \CC \) is 
a quasi-Gr\"{o}bner
category, then it is reasonable to
expect that the Hilbert series of
finitely generated \( \CC
\)-representations will be nice. For
example, finitely generated \( \FI
\)-representations have rational
Hilbert series.
\item Kruskal's tree Theorem is an 
important
part of the graph minor Theorem. The
category \( \TT \) is quasi-Gr\"{o}bner
because of Kruskal's tree Theorem. Is
there any category which is
quasi-Gr\"{o}bner because of the graph
minor Theorem? 
\end{enumerate}

\subsection{Acknowledgments}
I want to thank John Wiltshire-Gordon 
and Andrew Snowden for many useful 
discussions. This work would not have 
been possible without them. I want to 
thank Steven Sam for expressing 
interest in the categories \( \TT \) 
and \( \PT \) from an early stage, and 
suggesting ideas for future work.

\section{Rooted Trees} 
\label{subsec:rooted_trees}

In this section, we explain the terms 
and notation used throughout the 
paper. A {\bf tree} is a connected 
finite graph with no loops. A {\bf 
rooted tree} is a tree equipped 
with a root vertex.
In a rooted tree, we orient every edge 
towards the root vertex. When drawing 
rooted trees, the root vertex is at 
the bottom. Here is an example:
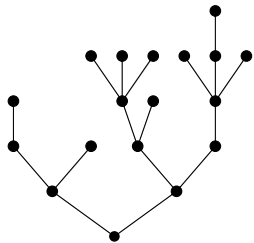
\[
\begin{tikzpicture}[
    level 1/.style={sibling distance=4em},
    level 2/.style={sibling distance=2.5em}, level distance=0.6cm,
    level 3/.style={sibling 
distance=1em}, level distance=0.6cm
    ]
\coordinate (root) {} [fill] circle (2pt) [grow=up]
child  { [fill] circle (2pt)
           child { [fill] circle (2pt)
                      child { [fill] circle (2pt)
                                 child { [fill] circle (2pt) }
                                 child { [fill] circle (2pt)
                                            child { [fill] circle (2pt) }
                                          }
                                 child { [fill] circle (2pt) }
                               }
                    }
           child { [fill] circle (2pt)
                      child { [fill] circle (2pt) }
                      child { [fill] circle (2pt)
                               child { [fill] circle (2pt) }
                               child { [fill] circle (2pt)}
                               child { [fill] circle (2pt) }
                               }
                    }
         }
child { [fill] circle (2pt)
           child { [fill] circle (2pt) }
           child { [fill] circle (2pt)
                      child { [fill] circle (2pt) }
                    }
         }
;
\end{tikzpicture}
\]
If \( v \) is a vertex, write \( 
\In(v) \) for the set of incoming 
edges. When we draw a picture of a 
rooted tree, we implicitly put an 
ordering on \( \In(v) \) for each 
vertex \( v \). A {\bf planar rooted 
tree} is a rooted tree equipped with a 
total ordering on \( \In(v) \) for 
each vertex \( v \). Given a rooted 
tree \( T \) we can build a partially 
ordered set as follows: The 
elements are vertices and the 
relations are generated by the edges pointing 
towards the root vertex. In other 
words, given vertices \( v,w \) we say 
that \( v \leq w \) if there is a 
downward path from \( v \) to \( w \). 
We call this order the {\bf tree order} on the 
vertices of \( T \). The root vertex 
is larger than all other vertices in 
the tree order. Let \( T \) be a planar 
rooted tree. We can totally order the 
vertices using a clockwise depth-first 
tree walk. This total ordering will be 
called the {\bf depth-first ordering} 
on the vertices and is denoted by \( 
\triangleleft \). When we say order 
embedding, we mean with respect to the 
tree order. 
\begin{align*}
\FTT &= \left\{ \parbox{8cm}{Objects 
are rooted trees and morphisms are 
order embeddings} 
\right\} \\ \\
\FPT &= \left\{ \parbox{8cm}{Objects 
are planar rooted trees and morphisms 
are order embeddings which also 
preserving the depth-first ordering on
  vertices} \right\} \\ \\
\TT &= \left\{ \parbox{8cm}{Objects 
are rooted trees and morphisms are 
order embeddings preserving the 
root} \right\}  \\ \\
\PT &= \left \{ \parbox{8cm}{Objects 
are planar rooted trees and morphisms 
are order embeddings that 
preserve the root and the depth-first 
ordering on vertices} \right\}
\end{align*}
The categories \( \TT \) and \( \PT \) 
are our main focus, but for many of 
the proofs, it is useful to work in \( 
\FTT \) and \( \FPT \). The 
morphisms in each of the above four 
categories must be injective on 
vertices. We can now state our main 
Theorem, from which Theorem 
\ref{T_is_noetherian} follows.
\begin{Theorem} \label{thm:main_theorem}
The category \( \PT \) is Gr\"{o}bner 
and the forgetful functor \( \PT
\to \TT \) is essentially surjective 
and has property (F).
\end{Theorem}
Theorem \ref{thm:main_theorem} says 
that \( \TT \) is quasi-Gr\"{o}bner. We 
refer the reader to \cite{SS2014}  
where the theory of Gr\"{o}bner  
categories is developed.

\section{A relative version of Kruskal's tree
  theorem} \label{sec:relative_version_kruskal}

We define a sequence of trees \(
B_1,B_2,B_3,B_4,\dots\) as follows: \( 
B_n \) is the graph with vertex set \( 
\{ * \} \cup \{ 1,\dots,n \} \) and 
edges \( (i,*) \). Diagrammatically, we 
have
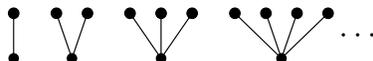
\[
\begin{tikzpicture}[baseline={([yshift=-.5ex]current bounding box.center)},
    level 1/.style={sibling distance=1em},
    level 2/.style={sibling distance=2.5em}, level distance=0.6cm
    ]
\coordinate (root) {} [fill] circle (2pt) [grow=up]
child {[fill] circle (2pt)};
\end{tikzpicture}
\quad
\begin{tikzpicture}[baseline={([yshift=-.5ex]current bounding box.center)},
    level 1/.style={sibling distance=1em},
    level 2/.style={sibling distance=2.5em}, level distance=0.6cm
    ]
\coordinate (root) {} [fill] circle (2pt) [grow=up]
child {[fill] circle (2pt)}
child {[fill] circle (2pt)};
\end{tikzpicture}
\quad
\begin{tikzpicture}[baseline={([yshift=-.5ex]current bounding box.center)},
    level 1/.style={sibling distance=1em},
    level 2/.style={sibling distance=2.5em}, level distance=0.6cm,
    ]
\coordinate (root) {} [fill] circle (2pt) [grow=up]
child {[fill] circle (2pt)}
child {[fill] circle (2pt)}
child {[fill] circle (2pt)};
\end{tikzpicture}
\quad
\begin{tikzpicture}[baseline={([yshift=-.5ex]current bounding box.center)},
    level 1/.style={sibling distance=1em},
    level 2/.style={sibling distance=2.5em}, level distance=0.6cm,
    ]
\coordinate (root) {} [fill] circle (2pt) [grow=up]
child {[fill] circle (2pt)}
child {[fill] circle (2pt)}
child {[fill] circle (2pt)}
child {[fill] circle (2pt)};
\end{tikzpicture}
\cdots
\]
These planar rooted trees form building blocks in the category \( \FPT
\).

\begin{Lemma} \label{lem:building_objects_from_colimits}
Let \( T \) be a planar rooted tree. Let \( v \) be a vertex of \( T \). Let \( T_v \) be the sub tree of \( T \) which contains everything above and including \( v \). Let \( T^v \) be the sub tree of \( T \) obtained by removing everything in \( T_v \) strictly above \( v \). Then we have the following pushout square in \( \FPT \):
\[
\begin{tikzcd}
T_v \arrow{r} & T \\
v \arrow{u} \arrow{r} & T^v \arrow{u}
\end{tikzcd}
\]
Here is an example of such a pushout square:
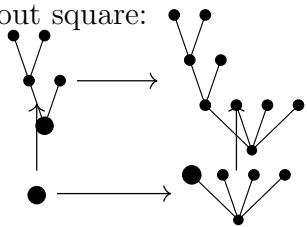
\[
\begin{tikzcd}
\begin{tikzpicture}[baseline={([yshift=-.5ex]current bounding box.center)},
    level 1/.style={sibling distance=1em},
    level 2/.style={sibling distance=1em}, level distance=0.6cm,
    ]
\coordinate (root) {} [fill] circle (3.5pt) [grow=up]
child {  [fill] circle (2pt)
         }
child { [fill] circle (2pt)
          child { [fill] circle (2pt)}
          child { [fill] circle (2pt)}
         };
\end{tikzpicture} \arrow{r} &
\begin{tikzpicture}[baseline={([yshift=-.5ex]current bounding box.center)},
    level 1/.style={sibling distance=1em},
    level 2/.style={sibling distance=1em}, level distance=0.6cm,
    level 3/.style={sibling distance=1em}, level distance=0.6cm,
    level 4/.style={sibling distance=1em}, level distance=0.6cm,
    ]
\coordinate (root) {} [fill] circle (2pt) [grow=up]
child {[fill] circle (2pt)}
child {[fill] circle (2pt)}
child {[fill] circle (2pt)}
child {[fill] circle (2pt)
          child {  [fill] circle (2pt)
                   }
          child { [fill] circle (2pt)
                    child { [fill] circle (2pt)}
                    child { [fill] circle (2pt)}
                   }
        };
\end{tikzpicture} \\
\begin{tikzpicture}[baseline={([yshift=-.5ex]current bounding box.center)},
    level 1/.style={sibling distance=1em},
    level 2/.style={sibling distance=2.5em}, level distance=0.6cm,
    ]
\coordinate (root) {} [fill] circle (3.5pt) [grow=up]
;
\end{tikzpicture} \arrow{u} \arrow{r} &
\begin{tikzpicture}[baseline={([yshift=-.5ex]current bounding box.center)},
    level 1/.style={sibling distance=1em},
    level 2/.style={sibling distance=2.5em}, level distance=0.6cm,
    ]
\coordinate (root) {} [fill] circle (2pt) [grow=up]
child {[fill] circle (2pt)}
child {[fill] circle (2pt)}
child {[fill] circle (2pt)}
child {[fill] circle (3.5pt)};
\end{tikzpicture}
\arrow{u}
\end{tikzcd}
\]
\end{Lemma}
\begin{proof}
To define a morphism \( T \to U \), we need to send edges in \( T \) to paths in \( U \) so that domains and codomains are preserved. Since every edge in \( T \) is contained in either \( T_v \) or \( T^v \), the lemma follows.
\end{proof}
\begin{Lemma} \label{lem:breaking_a_tree_up_at_the_root}
Assume that \( T \) is a planar rooted tree and \( v_1,\dots,v_n \) are the vertices with distance 1 from the root. Then \( T \) is a colimit of the following diagram (that we have only drawn for \( n =
3) \):
\[
\begin{tikzcd}
T_{v_1} & T_{v_2} & T_{v_3} & B_3 \\
v_1 \arrow{u} \arrow{urrr} & v_2  \arrow{u} \arrow{urr} & v_3
\arrow{u} \arrow{ur}
\end{tikzcd}
\]
\end{Lemma}
\begin{proof}
  This follows by repeated application of Lemma \ref{lem:building_objects_from_colimits}.
\end{proof}
\begin{Lemma} \label{lem:maps_from_Bn}
We have a natural isomorphism
\[
\Mor_{\FPT}(B_n,T) = \left\{ \parbox{7cm}{distinct vertices \(
    v,v_1,\dots,v_n \in T\) such that \(v_i \leq v \) in the tree
    order, the \( v_i \) are pairwise incomparable in the tree order
    and \( v_1 \triangleleft v_2 \triangleleft \dots \triangleleft v_n \) in the depth-first order} \right\}
\]
\end{Lemma}
Let \( T \) be a planar rooted tree. Define \( \PT_T \) to be the set of morphisms in \( \PT \) with domain \( T \). If \( f,g \in \PT_T \), we define \( f \leq g \) if there is a commutative diagram
\[
\begin{tikzcd}
T \arrow{r}{f} \arrow{rd}{g} & U \arrow{d} \\
& V
\end{tikzcd}
\]
in the category \( \PT \). 
Equivalently, \( f \leq g \) if there 
is a morphism \( h \) such that \( g = hf \). This is called the divisibility quasi-order on \( \PT_T \). Now we can state the relative version of Kruskal's tree Theorem:
\begin{Theorem} \label{thm:relative_kruskal}
The quasi-order on \( \PT_T \) is a well-quasi-order.
\end{Theorem}
The \( T = \bullet \) case is very 
similar to Kruskal's tree Theorem. Indeed, Lemma \ref{lem:T=single_vertex} is proved by Draisma in \cite{jD14}. We include a proof to establish notation and demonstrate the main proof technique in the easiest case.
\begin{Lemma} \label{lem:T=single_vertex}
Theorem \ref{thm:relative_kruskal} is true when \( T = \bullet \).
\end{Lemma}
\begin{proof}
We use the Nash-Williams theory of good/bad sequences that is explained in \cite[Chapter~12]{D}. Suppose that \( \PT_{\bullet} \) is not well-quasi-ordered. Given \( n \in \mathbb{N} \), assume inductively that we have chosen a sequence \( T_0,\dots,T_{n-1} \) of planar rooted trees such that some bad sequence of planar rooted trees starts with \( T_0,\dots,T_{n-1} \). Choose \( T_n \) with a minimal number of vertices such that some bad sequence starts \( T_0,T_1,\dots,T_n \). Then \( (T_n)_{n \in \mathbb{N}} \) is a bad sequence. We call \( (T_n) \) a minimal bad sequence. Let \( v_1,\dots,v_d \) be the vertices in \( T_n \) whose distance from the root is \( 1 \), ordered with respect to the depth-first ordering. Let \( A_n = T_{n,v_1} T_{n,v_2} \dots T_{n,v_n} \). If we think of each sequence \( A_n \) as a set, we can define \( A = \cup_n A_n \). We claim that \( A \) is well-quasi-ordered. Let \( (U_k) \) be a sequence in \( A \). Then \( U_k \in A_{n(k)} \), so we have a morphism \( U_k \to T_{n_k} \) in \( \FPT \). This morphism does not preserve the root, but we can modify what the morphism does on the root vertex in the following way:
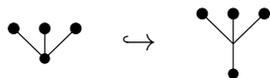
\[
\begin{tikzpicture}[baseline={([yshift=-.5ex]current bounding box.center)},
    level 1/.style={sibling distance=1em},
    level 2/.style={sibling distance=1em}, level distance=0.4cm,
    level 3/.style={sibling distance=1.5em}, level distance=0.4cm
    ]
\coordinate (root) {} [fill] circle (2pt) [grow=up]
child {[fill] circle (2pt)}
child {[fill] circle (2pt)}
child{ [fill] circle (2pt)};
\end{tikzpicture}
\quad
\hookrightarrow
\quad
\begin{tikzpicture}[baseline={([yshift=-.5ex]current bounding box.center)},
    level 1/.style={sibling distance=1em},
    level 2/.style={sibling distance=1em}, level distance=0.4cm,
    level 3/.style={sibling distance=1.5em}, level distance=0.4cm
    ]
\coordinate (root) {} [fill] circle (2pt) [grow=up]
child {
         child {[fill] circle (2pt)}
         child {[fill] circle (2pt)}
         child{ [fill] circle (2pt)}
        };
\end{tikzpicture}
\]
This allows us to convert \( U_k \to T_{n_k} \) into a morphism which witnesses \( U_k \leq T_{n_k} \). Choose \( p \) so that \( n(p) \) is the smallest element of \( \{ n(k) \} \). Then we have the following sequence \[ T_0,\dots,T_{n(p)-1},U_p,U_{p+1},\dots \] By the minimality of \( (T_n) \), it must have a good pair. If \( T_i \leq U_j \) then we have \( T_i \leq T_{n(j)} \). This is a contradiction because \( i < n(p) \leq n(j) \). Therefore there must be a good pair in \( (U_k) \). Since our choice of sequence in \( A \) was arbitrary, it follows that \( A \) is well-quasi-ordered. Consider the following sequence of words in \( A \):
\[ A_0,A_1,A_2,\dots \]
By Higman's lemma, we must have \( A_i \leq A_j \) for some \( i < j \). What this means is that there is an order preserving injection \( \phi : A_i \to A_j \) such that \( U \leq \phi(U) \) for each \( U \in A_i \). This gives us \( T_i \leq T_j \) which is a contradiction.
\end{proof}

\begin{Lemma} \label{lem:T=Bn}
Theorem \ref{thm:relative_kruskal} is true when \( T = B_n \).
\end{Lemma}
\begin{proof}
The proof is by induction on \( n \). The base case is \( n = 1 \). Elements in \( \PT_{B_1} \) are planar rooted trees with a distinguished non-root vertex and \( T \leq U \) if there is a morphism \( T \to U \) preserving the root and the distinguished non-root vertex. Choose a minimal bad sequence \( (T_n) \) in \( \PT_{B_1} \). Define \( A_n \) as in Lemma \ref{lem:T=single_vertex}. We can break the sequence \( A_n \) up as \( B_n U_n C_n \) where \( U_n \) is the tree containing the distinguished vertex, \( B_n \) is the sequence of trees coming before \( U_n \) and \( C_n \) is the sequence of trees coming after \( U_n \) in the depth first ordering. There are two cases we need to consider:
\begin{enumerate}
\item Firstly, suppose that for an infinite subsequence \( (U_{n_k}) \) of \( (U_n) \), the distinguished vertex in \( T_{n_k} \) is the root of \( U_{n_k} \). Consider the following sequence
\[ (B_{n_1},U_{n_1},C_{n_1}), (B_{n_2},U_{n_2},C_{n_2}), \dots \]
A product of well-quasi-orders is a well quasi-order. By lemma \ref{lem:T=single_vertex}, there must be a good pair \( (B_{n_i},U_{n_i},C_{n_i}) \leq (B_{n_j},U_{n_j},C_{n_j}) \) which gives us \( T_{n_i} \leq T_{n_j} \) in \( \PT_{B_1} \). This is a contradiction.
\item Secondly, suppose that for an infinite subsequence \( (U_{n_k}) \) of \( (U_n) \), the distinguished vertex in \( T_{n_k} \) is not the root of \( U_{n_k} \). The obvious morphism \( U_{n_k} \to T_{n_k} \) does not preserve roots, but we can use the same trick as in lemma \ref{lem:T=single_vertex} to get \( U_{n_k} \leq T_{n_k} \) in \( \PT_{B_1} \). Since we started with a minimal bad sequence, \( \{ U_{n_k} \} \) must be well-quasi-ordered, therefore the sequence
\[ (B_{n_1},U_{n_1},C_{n_1}), (B_{n_2},U_{n_2},C_{n_2}), \dots \]
must have a good pair \( (B_{n_i},U_{n_i},C_{n_i}) \leq (B_{n_j},U_{n_j},C_{n_j}) \) which gives us \( T_{n_i} \leq T_{n_j} \) in \( \PT_{B_1} \). This is a contradiction.
\end{enumerate}
One of these two cases must occur. Therefore we have proved that \( \PT_{B_1} \) is well-quasi-ordered. Now assume that \( \PT_{B_i} \) is well-quasi-ordered for \( i < n \). We prove that \( \PT_{B_n} \) is well-quasi-ordered. Elements of \( \PT_{B_n} \) are planar rooted trees with \( n \) distinguished non--root vertices \( v_1,\dots,v_n \) that are incomparable in the tree order and ordered in the depth-first order. We have \( T \leq U \) if there is a morphism \( T \to U \) in \( \FPT \) that preserves the root and the distinguished non root vertices. Assume that \( (T_n) \) is a minimal bad sequence in \( \PT_{B_n} \). As usual, form the sequence \( (A_n) \). Define \( \omega(A_n) \) as follows: replace each tree in \( A_n \) with the number of distinguished vertices of \( T_n \) it contains, then delete the zeros. By the pigeonhole principle
\[ \omega(A_1),\omega(A_2),\omega(A_3),\dots \]
must contain some sequence \( 
m_1,\dots,m_d \) an infinite number of times. Let \( (T_{n_k}) \) be the corresponding subsequence of \( (T_n) \). We must now consider two cases:
\begin{enumerate}
\item Suppose \( d = 1 \). Write \( A_{n_k} = B_{n_k} U_{n_k} C_{n_k} \) where \( U_{n_k} \) contains all of the distinguished vertices in \( T_{n_k} \). If there is an infinite subsequence where the root of \( U_{n_k} \) is not distinguished, then use a minimal bad sequence argument to get a contradiction. If there is an infinite subsequence where the root of \( U_{n_k} \) is distinguished, then use the induction hypothesis to get a contradiction.
\item If \( d > 1 \) then write
 \[ A_{n_k} = B_{n_k}^0 U_{n_k}^{1} 
B_{n_k}^1 \dots U_{n_k}^d B_{n_k}^d \]
where \( U_{n_k}^i \) has \( m_i \) of the distinguished vertices. Now use the induction hypothesis to get a contradiction.
\end{enumerate}
\end{proof}

\begin{proof}[Proof of Theorem \ref{thm:relative_kruskal}]
We induct on the number of vertices in 
\( T \). Lemma \ref{lem:T=Bn} is the base case. Choose a non--root vertex \( v \) in \( T \) that has valence \( \geq 2 \). Choose a sequence \( (\phi_n : T \to U_n) \) in \( \PT_T \). Then we get sequences \( \phi_{n,v} : T_v \to U_{n,\phi_n(v)} \) and \( \phi_n^v : T^v \to U_{n}^{\phi_n(v)} \) in \( \PT_{T_v} \) and \( \PT_{T^v} \) respectively. By induction, there must be a good pair \( (\phi_{i,v},\phi_{i}^v) \leq (\phi_{j,v},\phi_{j}^v) \). This induces \( \phi_i \leq \phi_j \) which completes the proof.
\end{proof}

\section{Proof of theorem \ref{thm:main_theorem}} \label{sec:proof_of_main_theorem}

In this section, we prove that \( \PT 
\) is a Gr\"{o}bner category and that 
the forgetful functor \( \PT \to \TT 
\) has property (F) and is essentially 
surjective. First, let us recall the 
definition of a Gr\"{o}bner category 
from \cite{SS2014}. Let \( \CC \) be a 
small directed category. Write \( \CC_x 
= \bigcup_{y} \Mor_{\CC}(x,y) \). If 
\( f : x \to y  \) and \( g : x \to z 
\) are elements of \( \CC_x \) then we 
write \( f \leq g \) if there is a 
commutative triangle
\[
\begin{tikzcd}
x \arrow{r}{f} \arrow{dr}[swap]{g} & y 
\arrow{d} \\
& z
\end{tikzcd}
\]
We call this quasi-order on \( \CC_x 
\) the {\bf divisibility order}. It is 
intrinsic to \( \CC \). An {\bf 
admissible order} on \( \CC_x \) is a 
well-order \( \preceq \) such that if 
\( f \preceq f' \) then \( gf \preceq 
gf' \) whenever this makes sense. 
Admissible orders are not intrinsic to 
\( \CC \): they are extra structure.
\begin{Definition} \label{def:grobner_category}
We call \( \CC \) {\bf Gr\"{o}bner} if each divisibility order \( \CC_x \) is a well-quasi-order and each \( \CC_x \) admits an admissible order.
\end{Definition}
Theorem \ref{thm:relative_kruskal} says 
that the divisibility order on \( 
\PT_T \) is a well-quasi-order. 
Therefore, to prove that \( \PT \) is 
Gr\"{o}bner, we need to construct 
admissible orders on each \( \PT_T \). 
Let \( T,U \) be planar rooted trees 
and choose a morphism \( \phi : T \to 
U \) in \( \PT \). If \( e \) is an 
edge in \( T \), label every edge in 
the path \( \phi(e) \) with the 
distance between \( {\rm target(e)} \) 
and \( {\rm root}(T) \) in \( T \). 
(edges point towards the root). Now we 
go on a clockwise depth-first tree 
walk  along \( U \) (depth-first tree 
walks are defined in \cite[chapter 
5]{Sv2}). As we are traveling, record 
the path as follows:
\begin{enumerate}
\item If we travel up an edge marked with an \( i \), write \( \up{i}
  \).
\item If we travel down an edge marked with an \( i \), write \(
  \down{i} \)
\item If we travel up an unmarked edge, write \( ( \)
\item If we travel down an unmarked edge, write \( ) \).
\end{enumerate}
The resulting string is called the {\bf Catalan word} of \( \phi
\).
\begin{Example} consider the map:
\[
\begin{tikzpicture}[baseline={([yshift=-.5ex]current bounding box.center)},
    level 1/.style={sibling distance=4em},
    level 2/.style={sibling distance=3em}, level distance=0.6cm,
    level 3/.style={sibling distance=1.5em}, level distance=0.6cm
    ]
\coordinate (root) {} [fill] circle (2pt) [grow=up]
child {[fill] circle (2pt)}
child {[fill] circle (2pt)
         child {[fill] circle (2pt)}
         };
\end{tikzpicture}
\quad
\hookrightarrow
\quad
\begin{tikzpicture}[baseline={([yshift=-.5ex]current bounding box.center)},
    level 1/.style={sibling distance=4em},
    level 2/.style={sibling distance=3em}, level distance=0.6cm,
    level 3/.style={sibling distance=1.5em}, level distance=0.6cm
    ]
\coordinate (root) {} [fill] circle (2pt) [grow=up]
child {
         child {[fill] circle (2pt)
                  child{
                          child
                          child
                          child
                          }
                  }
         child{ [fill] circle (2pt)
                  child
                  child {
                           child{ [fill] circle (2pt)
                                    child
                                    child
                                   }
                           }
                 }
        };
\end{tikzpicture}
\]
Its Catalan word is
\[
\up{0} \up{0} \up{1} \up{1} \up{} \down{} \up{} \down{} \down{1}
\down{1} \up{} \down{} \down{0} \up{0} \up{} \up{} \down{} \up{}
\down{} \up{} \down{} \down{} \down{0} \down{0}
\]
\end{Example}
If \( T = \bullet \) then we recover the the standard bijection between planar rooted trees and strings of balanced brackets which is described in \cite[chapter 5]{Sv2}.
\begin{Lemma} \label{lem:morphism_determined_by_Catalan_word}
The mapping \( \phi \mapsto \text{Catalan word} \) is injective.
\end{Lemma}
\begin{proof}
We can reproduce \( \phi \) from its Catalan word as follows. The top row of parentheses gives the target. The bottom row of numbers tells us how the domain is mapped in, and also gives the domain since all tree maps are fully faithful.
\end{proof}
We use Catalan words to equip each set 
\( \PT_T \) with an admissible order. Given a Catalan word, build the tuple \( (p,n) \) where \( p \) is the top row and \( n \) is the second row. We order the alphabets in the following way:
\begin{align*}
&) \prec ( \\
&- \prec 0 \prec 1 \prec 2 \prec 3 \prec \dots
\end{align*}
Order words in the parentheses alphabet using the length lexicographic ordering. Order words in \( \{ -,0,1,2,\dots \} \) using lexicographic ordering. Given two Catalan words \( (p,n), (p',n') \), define \( (p,n) \prec (p',n') \) if \( p \prec p' \) or \( p = p' \) and \( n \prec n' \).

\begin{Lemma} \label{lem:catalan_words_are_admissible}
Let \( f, g : T \to U \) be morphisms in \( \PT \) such that \( f \prec g \) with respect to the above Catalan word ordering. Let \( h : U \to V \) be a morphism. Then \( hf \prec hg \).
\end{Lemma}
\begin{proof}
First we interpret \( f \prec g \). 
When we go on a clockwise depth-first 
tree walk along \( U \), the first time 
we notice a difference in the edge 
labeling, the label for \( g \) is 
larger than the label for \( f \). Now 
go on a clockwise depth-first tree walk 
along \( V \) labeled by \( hf \) and 
\( hg \). The first difference that we 
notice is going to be induced by the 
difference we noticed on our walk 
along \( U \) and the label for \( hg 
\) will be bigger than the label for 
\( hf \) because the labels are mapped 
from \( U \).
\end{proof}
This completes the construction of 
admissible orders on each \( \PT_T \). Therefore we have proved that \( \PT \) is Gr\"{o}bner. To conclude the proof of Theorem \ref{thm:main_theorem}, we need to prove that the forgetful functor \( i : \PT \to \TT \) is essentially surjective and has property (F). First we recall the definition of property (F).
\begin{Definition} \label{def:property_F}
Let \( i : \CC' \to \CC \) be a functor. We say that \( i \) has {\bf property~(F)} if for each principal projective \( P_x = \mathbb{C} \CC(x,-) \) in \( \Rep(\CC) \), the \( \CC' \)-representation \( i^* P_x  \) is finitely generated.
\end{Definition}
Let \( J : \PT \to \TT \) be the functor which forgets the plane ordering. Since every rooted tree can be drawn on the plane it follows that \( J \) is essentially surjective. Let \( U \) be a rooted tree and \( V \) a planar rooted tree. Then we have
\[ \TT(U,J(V)) = \PT(U_1,V) \sqcup \PT(U_2,V) \sqcup \dots \sqcup
\PT(U_e,V) \]
where \( U_1,\dots,U_e \) are all the planar representations of \( U \). This implies that
\[ J^* P_U  = \bigoplus_{i=1}^e P_{U_i} \]
which proves that \( J \) has property (F).

\bibliographystyle{hplain} 
\bibliography{bibliography}

\end{document}